\documentclass[12pt]{article}
\usepackage{cmap}
\usepackage[utf8]{inputenc}
\usepackage[T1]{fontenc}
\usepackage[english]{babel}
\usepackage{amssymb,amsmath,amsfonts,amsthm,amscd,latexsym,indentfirst,verbatim,xcolor}
\usepackage{geometry}
\def\Fac{\mathcal F}
\def\Rel{\mathcal P}
\def\End{\textrm{End}}
\def\Span{\textrm{Span}}
\def\charr{\textrm{char}\,}
\def\id{\operatorname{id}}

\theoremstyle{plain}
\newtheorem{theorem}{Theorem}[section]
\newtheorem{lemma}[theorem]{Lemma}
\newtheorem{question}[theorem]{Question}
\newtheorem{proposition}[theorem]{Proposition}

\theoremstyle{definition}
\newtheorem{definition}[theorem]{Definition}

\begin{document}
\sloppy
\hfill{16N20, 16N40 (MSC2020)}

\begin{center}
{\Large
New results on nil rings}

\smallskip

Vsevolod Gubarev
\end{center}

\begin{abstract}
For every field~$\Bbbk$, we construct a non-nilpotent nil algebra~$A$
whose Lie algebra~$A^{(-)}$ is generated by ten elements.
For $\Bbbk = \mathbb Z_p$, its adjoint group can also be generated by ten elements.
Moreover, for every field~$\Bbbk$ algebraic over its prime subfield,
we construct a finitely presented algebra~$B$ over~$\Bbbk$ whose
Jacobson radical is not nil.
These results were found with the assistance of Aristotle (Harmonic).

{\it Keywords}:
nil ring, Jacobson radical, finitely presented algebra.
\end{abstract}

\section{Introduction}

We consider three questions concerning associative nil rings
and the Jacobson radical.
The first one concerns the nilpotency of a nil ring with a finitely generated
adjoint group. The second one asks whether a nil algebra is nilpotent
provided its Lie algebra is finitely generated.
The third question asks whether the Jacobson radical of a finitely presented
algebra is necessarily a nil ideal.
We construct counterexamples to all three statements under the indicated
conditions on the field.

All rings and algebras are associative and are not necessarily unital.
For a nil ring~$R$, the operation $a\circ b = a+b+ab$ defines the adjoint
group~$R^\circ$.
We denote by~$A^{(-)}$ the Lie algebra on the space~$A$
with the bracket $[a,b] = ab-ba$.
For a prime~$p$, put $\mathbb Z_p = \mathbb Z/p\mathbb Z$.

The structures $R^\circ$ and $R^{(-)}$ of a radical ring~$R$
are closely related. In 1955, S.~A.~Jennings proved~\cite{Jennings}
that $R^\circ$ is nilpotent if and only if $R^{(-)}$ is nilpotent.

\begin{question}[B.~Amberg, L.~Kazarin, 2007]\label{q:G}
Let~$R$ be a nil ring such that the group~$R^\circ$ is finitely generated.
Is~$R$ necessarily nilpotent?
\end{question}

\begin{question}[B.~Greenfeld, A.~Leroy, A.~Smoktunowicz,
M.~Ziembowski, 2015]\label{q:L}
Let~$A$ be a nil algebra over a field~$\Bbbk$ such that the Lie algebra~$A^{(-)}$
is finitely generated. Is $A$~necessarily nilpotent?
\end{question}

The first question was posed for nil algebras by B.~Amberg and L.~Kazarin
\cite[Question~9]{AK} and stated for nil rings in~\cite[Question~3]{Sm}.
The nilpotency of radical rings with a~two-generated adjoint group
was proved in~\cite{Sys}.
B.~Greenfeld obtained a similar result for boundedly generated
adjoint groups~\cite{Gr}.
Recall that a group~$G$ is boundedly generated if
$G = H_1\ldots H_k$ for some cyclic subgroups $H_1,\ldots,H_k\subseteq G$.

The second question was posed without restrictions on the characteristic
in~\cite[Question~37]{GLSZ} (2015).
It was repeated for fields of characteristic different from~$2$
in~\cite[Question~4]{Sm} (2018).
A.~Alahmadi and H.~Alsulami proved for fields of characteristic different from~2 that all positive terms of the derived series of~$A^{(-)}$
are finitely generated if the associative algebra~$A$ is generated
by finitely many nilpotent elements~\cite{AA}.
However, this does not imply that~$A^{(-)}$ itself is finitely generated.
J.~P.~Bell and B.~W.~Madill constructed an infinite-dimensional
graded-nil algebra of quadratic growth which is finitely generated
as a Lie algebra but is not a nil algebra~\cite[Theorem~6.1]{BM}.

\begin{theorem}\label{thm:nil}
For every field~$\Bbbk$, there exists a non-nilpotent nil algebra~$A$
over~$\Bbbk$ such that the Lie algebra~$A^{(-)}$ is generated by ten elements.
If $\Bbbk = \mathbb Z_p$, the algebra~$A$ can be chosen so that
its adjoint group~$A^\circ$ is also generated by ten elements.
\end{theorem}

Thus, the answers to Questions~\ref{q:G} and~\ref{q:L} are negative.
For the second question, no restriction on the characteristic is required.
Over an infinite field, the adjoint group of a~nonzero nil algebra
is not finitely generated~\cite{Gr}.

A.~Smoktunowicz stated the following question in her 2001 survey
\cite[Question~28]{Sm01}, attributing it to S.~A.~Amitsur
and referring to the 1999 survey of E.~R.~Puczylowski~\cite{Pu99}.

\begin{question}[S.~A.~Amitsur, survey of 2001]\label{q:amitsur}
Let~$A$ be a finitely presented associative algebra over a field~$\Bbbk$.
Is its Jacobson radical~$J(A)$ necessarily a nil ideal?
\end{question}

Over an uncountable field, the answer is positive.
Indeed, if $\dim A < |\Bbbk|$, then~$J(A)$ is a nil ideal~\cite{Am56}.

\begin{theorem}\label{thm:amitsur}
Let~$\Bbbk$ be a field algebraic over its prime subfield.
Then there exists a~finitely presented associative algebra~$B$
over~$\Bbbk$ whose Jacobson radical~$J(B)$ is not a nil ideal.
\end{theorem}

Theorem~\ref{thm:amitsur} gives a negative answer to
Question~\ref{q:amitsur} over the indicated fields.

The proof of Theorem~\ref{thm:nil} uses the Golod--Shafarevich
inequality~\cite{GS}, an algebra with a basis consisting of factors
of an infinite word~\cite{BM,BO}, and approximation by elements
of the adjoint group~\cite{Gr}.
For Theorem~\ref{thm:amitsur}, we construct a finitely presented
algebra~$A$ and an idempotent~$E\in A$ such that
$EAE\cong\Bbbk[t]_{(t)}$. Under this isomorphism, the image of~$t$
belongs to $J(EAE)\subseteq J(A)$ but is not nilpotent.

The presented results were found with the assistance of Aristotle (Harmonic)~\cite{Aristotle}.

\section{Preliminaries}\label{sec:prelim}

For an algebra~$R$ over~$\Bbbk$ without identity, denote by~$R^\#$
the algebra $\Bbbk1\oplus R$ obtained by adjoining an identity.
If~$R$ is a nil algebra, the correspondence $r\to 1+r$
induces a~group isomorphism $R^\circ\cong 1+R\subset R^\#$.

Let $F = \bigoplus_{n\geq1} F_n$ be the free algebra over~$\Bbbk$
on generators $t_1$,~$t_2$, where $F_n$ is the linear span
of the monomials of degree~$n$.
An ideal~$I$ of~$F$ generated by homogeneous elements is homogeneous,
that is, $I = \bigoplus_{n\geq1} I_n$, where $I_n = I\cap F_n$,
and $F/I = \bigoplus_{n\geq1} F_n/I_n$.

\begin{lemma}\label{lem:GSineq}
Let $\Rel\subseteq \bigcup_{j\geq2} F_j$,
$\sum_{f\in\Rel}(3/4)^{\deg f}\leq 1/2$,
and let~$I$ be the ideal of~$F$ generated by~$\Rel$.
Then $I_n\neq F_n$ for all $n\geq1$.
\end{lemma}

\begin{proof}
Put $r_j = |\Rel\cap F_j|$, $c_0 = 1$,
and $c_n = \dim F_n/I_n$ for $n\geq1$.
By~\cite{GS} (where the algebra~$F$ has an identity, which does not
affect the components~$I_n$, $n\geq1$), we have
$c_n\geq 2c_{n-1} - \sum_{j=2}^n r_j c_{n-j}$ for all $n\geq1$.
We prove by induction on~$n$ that $c_n\geq 4c_{n-1}/3$.
Indeed, since $c_{n-j}\leq (3/4)^{j-1}c_{n-1}$ for $1\leq j\leq n$, we obtain
$$
c_n \geq \left(2 - (4/3)\sum_{j = 2}^n r_j(3/4)^j\right)c_{n-1}
 \geq (2 - (4/3)(1/2))c_{n-1}
 = 4c_{n-1}/3.
$$
Hence $c_n\geq (4/3)^n > 0$.
\end{proof}

For an infinite word $x=x_0x_1x_2\ldots$, denote by~$\Fac_n$
the set of its factors of length~$n$, that is, the words
$x_ix_{i+1}\ldots x_{i+n-1}$ for $i\geq0$.
The cyclic shifts of a word $w=w_0\ldots w_{n-1}$ are the words
$w_i\ldots w_{n-1}w_0\ldots w_{i-1}$, $0\leq i<n$.

\begin{lemma}\label{lem:word}
There exist an infinite word~$x$ over the alphabet
$\{a,b,c,d,e\}$ containing all five letters and a constant~$C>0$
such that $|\Fac_n|\leq Cn$ for all $n\geq1$.
For every factor~$w$ of~$x$ of length at least~2,
some cyclic shift of~$w$ is not a factor of~$x$.
\end{lemma}

\begin{proof}
In~\cite[\S~2]{BO}, substitutions~$f$ and~$g$ are defined.
They assign to each $i\in\{0,1,2,3,4\}$ a word in
$\{0,1,2,3,4\}^*$ and a word in~$\{a,b,c,d,e\}^*$, respectively.
Both substitutions are extended to arbitrary words by concatenation
of the images of their letters.
By the formulas in~\cite[\S~2]{BO}, each~$f(i)$ begins with~01
and has length at least~4.
Thus, the words~$f^r(0)$, $r\geq0$, are successive initial segments
of an infinite word~$u=f^\omega(0)$.
Put $x=g(u)$. This word begins with~$abcdeacd$ and hence contains
all five letters.
By~\cite[Theorem~2.1]{BO}, every factor of~$x$ of length
at least~2 has a cyclic shift which is not a factor of~$x$.

By definition, $f(0) = 01203$ and $f(1) = 0124$.
Therefore, every word~$f^2(i)$ begins with~$f(0)f(1)$,
which contains all five digits.
Thus, the substitution~$f$ is primitive.
By~\cite[Theorem~10.4.12]{AS}, the number of distinct factors
of~$u$ of length~$n$ does not exceed~$C_0n$ for some constant~$C_0>0$.

By~\cite[\S~2]{BO}, all words~$g(i)$ are nonempty and have length at most~16.
Let~$w$ be a factor of length~$n$ of the word $x=g(u)$,
whose first letter occurs at position~$r$ in~$g(u_j)$, where
$u=u_0u_1\ldots$.
Then~$w$ is obtained from the word
$g(u_ju_{j+1}\ldots u_{j+n-1})$ by deleting its first $r-1$ letters
and taking the next~$n$ letters.
Thus,~$w$ is uniquely determined by the factor
$u_ju_{j+1}\ldots u_{j+n-1}$ of~$u$ and the integer~$r$.
Consequently, $|\Fac_n|\leq 16C_0n$ for all $n\geq1$.
\end{proof}

\section{Nil algebras and finite generation}\label{sec:nil}

Fix a field~$\Bbbk$, a word~$x$, and a constant~$C$ from Lemma~\ref{lem:word}.

\begin{definition}
Let~$M$ be the algebra over~$\Bbbk$ with basis
$\bigcup_{n \geq 1} \Fac_n$ in which the product of factors~$u$ and~$v$
is~$uv$ if $uv$ is a factor of~$x$, and zero otherwise.
Let~$M_n$ be the linear span of~$\Fac_n$.
Put $R = \bigoplus_{n \ge 1} R_n \subseteq M \otimes_\Bbbk F$,
where $R_n = M_n \otimes_\Bbbk F_n$,
$R_{\geq n} = \bigoplus_{j \geq n} R_j$, and
$h_{c,k} = c \otimes t_k$ for
$c \in \Fac_1 = \{a,b,c,d,e\}$, $k \in \{1,2\}$.
\end{definition}

The algebra~$M$ is associative, since both $(uv)w$ and $u(vw)$
are equal to~$uvw$ if $uvw$ is a factor of~$x$, and are zero otherwise.
For every nonempty factor~$w$, we have $w^2 = 0$ in~$M$.
If $|w| = 1$, then $ww$ contains all its cyclic shifts,
contrary to Lemma~\ref{lem:word}.
If $|w|\geq2$, then $ww$ contains all cyclic shifts of~$w$.
In particular, $h_{c,k}^2 = 0$.
Denote by~$H$ the subgroup of the group of invertible elements of~$R^\#$
generated by the elements $1 + \lambda h_{c,k}$, $\lambda\in\Bbbk$.
Their inverses are $1 - \lambda h_{c,k}$, so $H \subseteq 1 + R$.

The first assertion of the following lemma essentially repeats
Proposition~6.2 of~\cite{BM}.

\begin{lemma}\label{lem:lie}
The Lie algebra~$R^{(-)}$ is generated by the ten elements~$h_{c,k}$.
For any $z\in R$ and $n\geq1$, there exists $u\in H$
such that $u - 1 - z\in R_{\geq n}$.
\end{lemma}

\begin{proof}
The space~$R_n$ has a basis consisting of the elements
$b = w\otimes y$, where $w\in\Fac_n$ and~$y\in F_n$ is a monomial.
We prove simultaneously by induction on~$n$ that every such~$b$
belongs to the Lie algebra generated by~$h_{c,k}$, and that
$1 + \lambda b\in H$ for all $\lambda\in\Bbbk$.
For $n = 1$, this follows from the definitions.
For $n\geq2$, by Lemma~\ref{lem:word}, there exists a decomposition
$w = uv$ into nonempty words such that $vu = 0$ in~$M$.
Write $y = y'y''$ with $\deg y' = |u|$.
For $b_1 = u\otimes y'$ and $b_2 = v\otimes y''$, we have
$b_1^2 = b_2^2 = b_2b_1 = 0$.
Hence $b = [b_1,b_2]$ and
\begin{equation}\label{eq:exactcomm}
1 + \lambda b
 = (1 + \lambda b_1)(1 + b_2)(1 - \lambda b_1)(1 - b_2) \in H.
\end{equation}
The inclusion in~$H$ follows by the induction hypothesis.

We prove the second assertion by induction on~$n$, as in
\cite[Lemma~3.1]{Gr}. For $n = 1$, we may take $u = 1$.
Suppose that $e = z - (u - 1) \in R_{\geq n}$ and let
$e_n = \sum_b \lambda_b b$ be its homogeneous component of degree~$n$.
By~\eqref{eq:exactcomm} and the case $n = 1$ proved above,
the product $v = \prod_b(1 + \lambda_b b)$, taken in any order,
belongs to~$H$, and $v - 1 - e_n \in R_{\geq n+1}$.
Consequently,
\begin{equation*}
vu - 1 - z
 = (v - 1 - e_n) - (e - e_n) + (v - 1)(u - 1)
 \in R_{\geq n+1}. \qedhere
\end{equation*}
\end{proof}

For a homogeneous ideal $I\subseteq F$, denote by
$\varphi_I\colon R \to \bigoplus_{n\geq1} M_n \otimes_\Bbbk (F_n/I_n)$
the surjective algebra homomorphism defined by
$\varphi_I(w\otimes f) = w\otimes(f + I_n)$ for $w\in M_n$, $f\in F_n$.
Every element $r\in R_{\geq N}$ has a unique expression
$r = \sum_{j\geq N}\sum_{w\in\Fac_j}w\otimes f_w$,
where $f_w\in F_j$.
By Lemma~\ref{lem:word}, the sum of $(3/4)^{\deg f}$
over all nonzero coefficients~$f$ of~$r$ does not exceed
$\sigma(N) = \sum_{j\geq N} Cj(3/4)^j$.
For every $q\geq0$, the number $(N + 1)^q\sigma(N)$ does not exceed
the tail $\sum_{j\geq N} C(j + 1)^{q+1}(3/4)^j$
of a convergent series. Therefore,
\begin{equation}\label{eq:tail}
\lim\limits_{N \to \infty}(N + 1)^q\sigma(N) = 0.
\end{equation}

\begin{proof}[Proof of Theorem~\ref{thm:nil}]
We construct a quotient of~$R$ in which sufficiently high powers
of all elements are zero.
For $\Bbbk = \mathbb Z_p$, we additionally impose relations
of the form $u - 1 - z = 0$, where $u\in H$ approximates~$1 + z$.
We choose these relations in sufficiently high degrees so that
the quotient remains non-nilpotent by the Golod--Shafarevich inequality.

For $d\geq1$, put $V_d = \bigoplus_{j=1}^d R_j$
and $q_d = \dim V_d$.
For any $N\geq1$, the linear span of the powers $z^N$,
$z\in V_d$, has dimension at most $(N + 1)^{q_d}$.
Indeed, choose a basis $e_1,\ldots,e_{q_d}$ of~$V_d$ and consider
$P(\lambda_1,\ldots,\lambda_{q_d})
 = (\lambda_1e_1+\ldots+\lambda_{q_d}e_{q_d})^N$
as a polynomial in commuting variables~$\lambda_i$ with coefficients in~$R$.
Substituting $\lambda_i = a_i\in\Bbbk$, we obtain
$P(a_1,\ldots,a_{q_d}) = z^N$, where
$z = \sum_{i=1}^{q_d}a_ie_i$.
Thus, $\dim\Span\{z^N\mid z\in V_d\}\leq (N + 1)^{q_d}$.

By~\eqref{eq:tail}, choose an integer $N_d\geq2$ such that
$2(N_d + 1)^{q_d}\sigma(N_d)\leq2^{-d-1}$.
For $\Bbbk = \mathbb Z_p$, also require $N_d\geq p$.
In this case, for each $z\in V_d$, by Lemma~\ref{lem:lie},
choose $u_{d,z}\in H$ such that
$u_{d,z} - 1 - z \in R_{\geq N_d}$.
Define
$$
E_d = \begin{cases}
\Span\{z^{N_d}\mid z\in V_d\}, & \Bbbk\neq\mathbb Z_p,\\[4pt]
\Span\{z^{N_d},\ u_{d,z}-1-z\mid z\in V_d\},
 & \Bbbk=\mathbb Z_p.
\end{cases}
$$
In both cases, $E_d\subseteq R_{\geq N_d}$ and
$\dim E_d\leq 2(N_d + 1)^{q_d}$.
For every $d\geq1$, choose a basis~$\mathcal B_d$ of~$E_d$.
Every $e\in\mathcal B_d$ has a unique expression
$e = \sum_{j\geq N_d}\sum_{w\in\Fac_j}w\otimes f_{e,w}$,
where $f_{e,w}\in F_j$.
Put
$$
\Rel = \{f_{e,w}\mid d\geq1,\,e\in\mathcal B_d,\,
 j\geq N_d,\,w\in\Fac_j\}\setminus\{0\},
\quad I = \langle\Rel\rangle \triangleleft F.
$$
All elements of~$\Rel$ are homogeneous and have degree at least~2.
Hence the ideal~$I$ is homogeneous.
To apply Lemma~\ref{lem:GSineq}, estimate
$\sum_{f\in\Rel}(3/4)^{\deg f}$.
For fixed $d$ and~$j$, there are at most
$|\mathcal B_d|\cdot|\Fac_j|$ nonzero coefficients~$f_{e,w}$
of degree~$j$. Therefore,
$$
\sum_{f\in\Rel}(3/4)^{\deg f}
\leq
\sum_{d\geq1}|\mathcal B_d|
\sum_{j\geq N_d}|\Fac_j|(3/4)^j
\leq
\sum_{d\geq1}2(N_d+1)^{q_d}\sigma(N_d)
\leq
\sum_{d\geq1}2^{-d-1}=1/2.
$$

By the definition of~$I$, we have $\varphi_I(E_d) = 0$
for all $d\geq1$.
Put $A = \varphi_I(R)$ and $g_{c,k} = \varphi_I(h_{c,k})$.
By Lemma~\ref{lem:GSineq}, for every $n\geq1$ there exist
$k_1,\ldots,k_n\in\{1,2\}$ such that
$(t_{k_1}+I)\ldots(t_{k_n}+I)\neq0$.
Since $x_0\ldots x_{n-1}\in\Fac_n$, it follows that
$$
g_{x_0,k_1}\ldots g_{x_{n-1},k_n}
 = (x_0\ldots x_{n-1})\otimes
 ((t_{k_1}+I)\ldots(t_{k_n}+I))\neq0.
$$
Thus, the algebra~$A$ is not nilpotent.
For any $z\in V_d$, we have $z^{N_d}\in E_d$ and hence
$\varphi_I(z)^{N_d} = 0$.
Since $R = \bigcup_{d\geq1} V_d$, the algebra~$A$ is nil.
By Lemma~\ref{lem:lie}, the Lie algebra~$A^{(-)}$
is generated by the elements~$g_{c,k}$.

If $\Bbbk = \mathbb Z_p$, the identity
$(1 + h_{c,k})^r = 1 + rh_{c,k}$, $0\leq r<p$,
implies that~$H$ is generated by the elements $1 + h_{c,k}$.
The homomorphism $\pi\colon R^\#\to A^\#$,
$\pi(\lambda + r) = \lambda + \varphi_I(r)$,
transfers~$H$ to the subgroup generated by $1 + g_{c,k}$,
and~$u_{d,z}$ to $1 + \varphi_I(z)$, since
$\varphi_I(E_d) = 0$.
Every element of~$R$ belongs to some~$V_d$, so
$\pi(H) = 1 + A$.
By the isomorphism $A^\circ\cong1+A$, the group~$A^\circ$
is generated by the ten elements~$g_{c,k}$.
\end{proof}

\section{Finitely presented algebras}
\label{sec:amitsur}

The main idea is to obtain inverses of an infinite family of polynomials
using one invertible element.
To this end, we index the polynomials by finite words and use
finitely many generators of an algebra to pass between these words.

We construct the algebra from Theorem~\ref{thm:amitsur}.
Let~$\Bbbk$ be a field and let~$m\geq1$ be an integer.
Put $\Sigma = \{1,\ldots,m\}$, and denote by~$\Sigma^*$
the set of words over~$\Sigma$, including the empty word~$\varnothing$.
For $i\in\Sigma$, denote by~$i^r$ the word consisting of $r$ letters~$i$.

\begin{definition}\label{d:algebra}
Let $\Phi = ((a_1,b_1),\ldots,(a_m,b_m))$, where
$a_i,b_i\in\Bbbk[t]$.
Define the unital algebra~$B_\Phi$ over~$\Bbbk$ by generators
$S_1,\ldots,S_m,T_1,\ldots,T_m,\theta,D,R$
and relations
\begin{equation}\label{eq:def-rel}
T_iS_j = \delta_{ij},\
\theta = E\theta E + \sum_{i=1}^m S_i\theta T_i,\
D = E + \sum_{i=1}^m S_i(a_i(\theta)D+b_i(\theta))T_i,\
RD = DR = 1,
\end{equation}
where $1\leq i,j\leq m$ and $E = 1 - \sum_{i=1}^m S_iT_i$.
\end{definition}

For a given tuple~$\Phi$, define polynomials
$p_w\in\Bbbk[t]$, $w\in\Sigma^*$, recursively by
$p_\varnothing = 1$ and $p_{iw} = a_ip_w + b_i$.

We explain these generators and relations by means of a representation.
Consider
$L = \Bbbk[t]_{(t)}
 = \{g/s\in\Bbbk(t)\mid g,s\in\Bbbk[t],\,s(0)\neq0\}$.
Suppose that $p_w(0)\neq0$ for all $w\in\Sigma^*$.
On the free left $L$-module~$V$ with basis
$\{v_w\mid w\in\Sigma^*\}$, define the action of the generators
of~$B_\Phi$ by the following $L$-linear operators:
$$
S_iv_w = v_{iw},\quad
T_iv_{jw} = \delta_{ij}v_w,\quad
T_iv_\varnothing = 0,\quad
\theta v_w = tv_w,\quad
Dv_w = p_wv_w,\quad
Rv_w = p_w^{-1}v_w.
$$
We verify the defining relations.
The equalities $T_iS_j = \delta_{ij}$ hold on each basic vector.
The formulas for~$S_i,T_i$ imply
$Ev_\varnothing = v_\varnothing$ and $Ev_{jw} = 0$.
Thus,~$E$ acts as the projection onto~$Lv_\varnothing$.
Since $\theta = t\id_V$, we have
$E\theta E + \sum_{i=1}^m S_i\theta T_i
 = t\bigl(E + \sum_{i=1}^m S_iT_i\bigr) = \theta$.
The third relation in~\eqref{eq:def-rel} holds on~$v_\varnothing$
since $p_\varnothing = 1$.
The action of its right-hand side on~$v_{jw}$ gives
$$
S_j(a_j(\theta)D+b_j(\theta))v_w
 = (a_jp_w+b_j)v_{jw} = p_{jw}v_{jw} = Dv_{jw}.
$$
Finally, $p_w(0)\neq0$, so $p_w$ is invertible in~$L$,
and $RD = DR = 1$ on every basic vector.
Thus,~$V$ is a unital left $B_\Phi$-module,
and the action is $L$-linear.
Denote the corresponding algebra homomorphism by
$\rho\colon B_\Phi\to\End_L(V)$.

To prove Theorem~\ref{thm:amitsur}, take the following tuples
of pairs of polynomials:
\begin{equation}\label{eq:Phi-cases}
\Phi = \begin{cases}
((t,1)), & \charr\Bbbk = p > 0,\\
((1,1),(1,-t),(t,1)), & \charr\Bbbk = 0.
\end{cases}
\end{equation}
The corresponding polynomials~$p_w$ satisfy
$p_w(0)\neq0$ for all $w\in\Sigma^*$.
Indeed, if $\charr\Bbbk = p > 0$, then $p_w(0) = 1$
for all~$w$.
If $\charr\Bbbk = 0$, the formulas
$p_{1w} = p_w+1$, $p_{2w} = p_w-t$, and
$p_{3w} = tp_w+1$ show that $p_w(0)\geq1$
for every~$w\in \Sigma^*$.

\begin{proposition}\label{p:families}
Let~$\Bbbk$ be a field algebraic over its prime subfield,
and let~$\Phi$ be defined by~\eqref{eq:Phi-cases}.
For every irreducible polynomial~$q\in\Bbbk[t]$
with $q(0)\neq0$, there exists $w\in\Sigma^*$
such that $q$ divides~$p_w$.
\end{proposition}

\begin{proof}
Let $q\in\Bbbk[t]$ be an irreducible polynomial with $q(0)\neq0$.
Every root~$\alpha$ of~$q$ is then algebraic over the prime subfield
of~$\Bbbk$.
If $p_w(\alpha) = 0$, then $q\mid p_w$.
Thus, it suffices to show that every nonzero element algebraic
over the prime subfield of~$\Bbbk$ is a root of some~$p_w$.

Let $\charr\Bbbk = p > 0$.
Every nonzero element~$\alpha$ algebraic over~$\mathbb Z_p$
has finite multiplicative order~$h$.
Hence,
$p_{1^{ph-1}}(\alpha)
 = 1 + \alpha + \ldots + \alpha^{ph-1}
 = p(1 + \alpha + \ldots + \alpha^{h-1}) = 0$.

Let $\charr\Bbbk = 0$.
First we prove that every polynomial $\gamma\in\mathbb Z[t]$
such that $\gamma(0)>0$ coincides with some~$p_w$.
We proceed by induction on $\deg\gamma$.
For a constant $c\geq1$, we have $c = p_{1^{c-1}}$.
Let $\deg\gamma\geq1$ and write
$\gamma = t\delta + c$, where
$c = \gamma(0)\in\mathbb Z_{>0}$ and $\delta\in\mathbb Z[t]$.
Choose $k\geq0$ such that $\delta(0)+k\geq1$.
By induction, there exists $w\in\Sigma^*$ such that
$\delta+k = p_w$. Hence,
$$
\gamma = t(p_w)+1-kt+(c-1) = p_{1^{c-1}2^k3w}.
$$
For every nonzero algebraic number~$\alpha$, there exists
$\nu\in\mathbb Z[t]$ such that
$\nu(\alpha) = 0$ and $\nu(0)>0$.
Therefore, $\nu = p_w$ for some~$w$, and $p_w(\alpha) = 0$.
\end{proof}

\begin{proof}[Proof of Theorem~\ref{thm:amitsur}]
Put $B = B_\Phi$.
For a word $w = i_1\ldots i_n\in\Sigma^*$, denote
$S_w = S_{i_1}\ldots S_{i_n}$,
$T_w = T_{i_n}\ldots T_{i_1}$, and
$S_\varnothing = T_\varnothing = 1$.
Then $T_wS_w = 1$.
The relations $T_iS_j = \delta_{ij}$ imply
$ES_i = T_iE = 0$ and $E^2 = E$.
The second relation in~\eqref{eq:def-rel} gives
$\theta S_i = S_i\theta$ and $T_i\theta = \theta T_i$.
Thus,~$\theta$ commutes with~$E$, $S_w$, and~$T_w$.
The third relation in~\eqref{eq:def-rel} implies
$DE = ED = E$,
$DS_i = S_i(a_i(\theta)D+b_i(\theta))$, and
$T_iD = (a_i(\theta)D+b_i(\theta))T_i$.
We prove by induction on the length of~$w$ that
\begin{equation}\label{eq:values}
DS_wE = S_wp_w(\theta)E,\quad ET_wD = p_w(\theta)ET_w.
\end{equation}
For $w = \varnothing$, these are the relations $DE = ED = E$.
Further, we obtain
$$
DS_{iw}E = S_i(a_i(\theta)D+b_i(\theta))S_wE
 = S_{iw}p_{iw}(\theta)E,
$$
and similarly
$ET_{iw}D = ET_w(a_i(\theta)D+b_i(\theta))T_i
 = p_{iw}(\theta)ET_{iw}$.
Put $c_w = ET_wRS_wE$.
Applying~\eqref{eq:values}, we obtain
$$
p_w(\theta)Ec_w = ET_wDRS_wE = E,\quad
c_wp_w(\theta)E = ET_wRDS_wE = E.
$$

Consider the subalgebra $EBE\subset B$, whose identity is~$E$.
By Proposition~\ref{p:families}, for every irreducible polynomial
$q\in\Bbbk[t]$ with nonzero constant term,
there exist $w\in\Sigma^*$ and $h\in\Bbbk[t]$
such that $p_w = qh$.
Put $a = q(\theta)E$ and $b = h(\theta)E$.
These elements commute, and their product
$ab = p_w(\theta)E$ has a two-sided inverse~$c_w$.
Hence $a(bc_w) = E$ and $(c_wb)a = E$.
In an associative ring, the left and right inverses of an element,
if they exist, coincide.
Thus, $bc_w = c_wb$, and $q(\theta)E$ is invertible in~$EBE$
with inverse $h(\theta)Ec_w$.
Every polynomial $s\in\Bbbk[t]$ such that $s(0)\neq0$
is a product of a nonzero constant and irreducible polynomials
with nonzero constant terms.
Therefore, $s(\theta)E$ is also invertible in~$EBE$.

The correspondence $g\to g(\theta)E$ defines a homomorphism
of unital algebras $\Bbbk[t]\to EBE$.
Since all polynomials with nonzero constant term are sent
to invertible elements, it extends to a homomorphism
$$
\varphi\colon L = \Bbbk[t]_{(t)}\to EBE,\quad
\varphi(g/s) = g(\theta)E(s(\theta)E)^{-1}.
$$
In particular, $\varphi(1) = E$ and $\varphi(t) = \theta E$.

In the representation~$\rho$, the element~$E$ acts as
the projection onto the free $L$-module
$EV = Lv_\varnothing$ of rank~1.
The module~$EV$ is invariant under every $a\in EBE$.
Since it is a free $L$-module of rank~1, each such~$a$
acts on~$EV$ as multiplication by an element of~$L$.
This defines a homomorphism $\varepsilon\colon EBE\to L$
such that $\varepsilon(\theta E) = t$.
Consequently, $\varepsilon(\varphi(g/s)) = g/s$,
so $\varepsilon\varphi = \id_L$.
In particular, $\varphi$ is injective.

Put $U = \sum_{w\in\Sigma^*}S_w\varphi(L)\subseteq BE$.
Note that $E = \varphi(1)\in U$.
The space~$U$ is invariant under left multiplication by~$S_i$
and~$T_i$. This follows from $T_iS_j = \delta_{ij}$ and $T_iE = 0$.
It is also invariant under left multiplication by~$\theta$ and~$D$.
Indeed, for any $h\in L$ and $w\in\Sigma^*$,
by~\eqref{eq:values} we have
$$
\theta S_w\varphi(h) = S_w\varphi(th),\quad
DS_w\varphi(h) = S_w\varphi(p_wh).
$$
Every polynomial~$p_w$ is invertible in~$L$ since $p_w(0)\neq0$.
Thus, $RD = 1$ yields
$RS_w\varphi(h)
 = RDS_w\varphi(p_w^{-1}h)
 = S_w\varphi(p_w^{-1}h)$.
Therefore, $U$ is invariant under left multiplication by all
generators of~$B$.
Since $E\in U$ and $U\subseteq BE$, we obtain $U = BE$.
The equality $ES_w = 0$ for all
$w\in\Sigma^*\setminus\{\varnothing\}$ implies
$EBE = EU = \varphi(L)$.
Together with the injectivity of~$\varphi$, this implies
$EBE\cong L$.

For the local ring~$L$, we have $J(L) = tL$.
Therefore,
$\theta E = \varphi(t)\in J(EBE) = EJ(B)E\subseteq J(B)$.
However, $(\theta E)^n = \varphi(t^n)\neq0$ for every $n\geq1$,
since $\varphi$ is injective and $t^n\neq0$ in~$L$.
\end{proof}

The presence of an identity in~$B$ is not essential
for the answer to Question~\ref{q:amitsur}.
Suppose that~$B$ is presented in the variety of unital
associative algebras by generators $x_1,\ldots,x_n$
and relations $f_j = 0$.
It can also be presented in the variety of associative algebras
by adjoining a generator~$u$ and relations
$u^2 = u$, $ux_i = x_iu = x_i$, and by replacing
the constant term $c\cdot1$ in every~$f_j$ with~$cu$.

\vspace{-0.1cm}

\noindent Vsevolod Gubarev \\
Sobolev Institute of Mathematics \\
Acad. Koptyug ave. 4, 630090 Novosibirsk, Russia \\
Novosibirsk State University \\
Pirogova str. 1, 630090 Novosibirsk, Russia \\
e-mail: wsewolod89@gmail.com

\end{document}